\documentclass[a4paper,12pt]{article}
\usepackage{amsmath,amsfonts,amsthm,amssymb}
\usepackage[alphabetic]{amsrefs}
\usepackage[all,cmtip]{xy}
\usepackage{graphicx}
\usepackage[all]{xypic}
\usepackage{tikz}

\usepackage{todonotes}
\usepackage{paralist}
\usepackage{cleveref}

\IfFileExists{enumitem.sty}{\usepackage{enumitem}}{}

\long\def\symbolfootnote[#1]#2{\begingroup%
\def\thefootnote{\fnsymbol{footnote}}\footnote[#1]{#2}\endgroup}

\newtheorem{theorem}{Theorem}
\newtheorem{proposition}{Proposition}[section]
\newtheorem{corollary}[proposition]{Corollary}

\newtheorem{theoremothers}{Theorem}

\theoremstyle{definition}

\newtheorem{remark}[proposition]{Remark}
\newtheorem{definition}[proposition]{Definition}

\newtheorem{example}[proposition]{Example}

\renewcommand{\qed}{\quad\hskip0pt\null\hfill$\square$\par\vspace*{3mm}}
\renewcommand{\theequation}{\thesection.\arabic{equation}} 

\renewcommand{\theenumi}{\roman{enumi}}
\renewcommand{\labelenumi}{(\theenumi)}

\def\Q{\mathbb{Q}}
\def\Z{\mathbb{Z}}
\def\R{\mathbb{R}}
\def\N{\mathbb{N}}
\def\C{\mathbb{C}}
\def\I{\mathrm{Id}}
\def\PP{\mathbb{P}}
\def\e{e_1}
\def\f{e_2}
\newcommand{\gl}{\mathrm{glue}}

\newcommand{\compl}[1]{\overline{#1}}

\newcommand{\sltwo}{\mbox{SL}_2}
\newcommand{\gltwo}{\mbox{GL}_2}

\newcommand{\RR}{\mathbb{R}}
\newcommand{\NN}{\mathbb{N}}
\newcommand{\ZZ}{\mathbb{Z}}
\newcommand{\CC}{\mathbb{C}}
\newcommand{\QQ}{\mathbb{Q}}
\newcommand{\bpm}{\begin{pmatrix}}
\newcommand{\epm}{\end{pmatrix}}
\newcommand{\Ksc}{K_{\mbox{\footnotesize sc}}}
\newcommand{\Kcr}{K_{\mbox{\footnotesize cr}}}
\newcommand{\Khol}{K_{\mbox{\footnotesize hol}}}
\newcommand{\Ktr}{K_{\mbox{\footnotesize tr}}}
\newcommand{\out}{\mbox{Out}}

\newcommand{\Sigmafin}{\Sigma_{\mbox{\scriptsize fin}}}
\newcommand{\Sigmainf}{\Sigma_{\inf}}
\newcommand{\semicompl}[1]{\widehat{#1}}

\newcommand{\emKsc}{K_{\mbox{\em \footnotesize sc}}}
\newcommand{\emKcr}{K_{\mbox{\em \footnotesize cr}}}
\newcommand{\emKhol}{K_{\mbox{\em \footnotesize hol}}}
\newcommand{\emKtr}{K_{\mbox{\em \footnotesize tr}}}
\newcommand{\CCC}{\mathcal{C}}
\begin{document}

\begin{center}
\large\bfseries On the geometry and arithmetic of infinite translation surfaces
\end{center}

\begin{center}\bf
Ferr\'an Valdez$^a$\symbolfootnote[1]{Partially supported by Sonderforschungsbereich/Transregio 45, PAPIIT-UNAM  IN103411 \& IB100212, IACOD-IA100511 and CONACYT CB-2009-01 127991.}
\& Gabriela Weitze-Schmith\"{u}sen$^b$\symbolfootnote[2]{Supported 
within the program RiSC (sponsored by the 
Karlsruhe Institute of Technology and the MWK Baden-W\"urttemberg) and the  
Juniorprofessorenprogramm (sponsored by the MWK  Baden-W\"urttemberg)}
\end{center}

\begin{center}\it
$^a$
 Instituto de Matem\'aticas, U.N.A.M.\\
Campus Morelia, Morelia, Michoac\'an, M\'exico \\
\emph{e-mail:}\texttt{ferran@matmor.unam.mx}
\end{center}

\begin{center}\it
$^b$
Institute of Algebra and Geometry, Department of Mathematics, \\
Karlsruhe Institute of Technology (KIT), D-76128 Karlsruhe, Germany \\
\emph{e-mail:}\texttt{weitze-schmithuesen@kit.edu}

\end{center}

\emph{In honorem professoris Xavier G\'omez-Mont sexagesimum annum complentis.}

\begin{abstract}
\noindent
Precompact translation surfaces, {\emph i.e.} closed surfaces which carry
 a translation atlas outside of finitely many finite angle cone points,  
have been intensively studied for about 25 years now.
About 5 years ago the attention was also 
drawn to general translation surfaces.
In this case
the underlying surface can have infinite genus,
 the number of finite angle cone points of the translation structure can be infinite,
and there can be singularities which are not finite angle cone points. 
There are only a few invariants one classically 
associates with precompact translation surfaces, among them certain number fields,
\emph{i.e.} fields which are finite extensions of $\QQ$. These fields
are closely related to each other;
they are often even equal.
We prove by constructing explicit examples that most of the classical results for the fields associated 
with precompact translation surfaces fail in the realm of general
translation surfaces
and investigate the relations among these fields. 
A very special class of translation surfaces 
are so called square-tiled surfaces or origamis.
We give a characterisation for  infinite origamis.
\end{abstract}




\section{Introduction}
\indent Let $S$ be a translation surface, in the sense of Thurston
\cite{Thu}, and denote by $\overline{S}$ the metric completion with
respect to its natural translation invariant flat metric. $S$ is called 
{\em precompact} if $\overline{S}$ is homeomorphic to a compact surface. 
We call translation surfaces {\em origamis}, if they are obtained from 
gluing copies of the Euclidean unit square 
along parallel edges by translations; see \Cref{def-origamis}. They are precompact
translation surfaces if and only if the number of copies is finite. An important invariant associated with a  
translation surface $S$
is the {\em Veech group} $\Gamma(S)$ formed by the differentials  of 
affine diffeomorphisms of $S$ that preserve orientation;
as further invariants one considers 
the {\em trace field} $\Ktr(S)$, the {\em holonomy field} $\Khol(S)$, 
the {\em field
of cross ratios of saddle connections} $\Kcr(S)$ and the {\em field of saddle connections}
$\Ksc(S)$; compare \Cref{vgroup} and \Cref{definition-fields}. 
For precompact surfaces we have the following characterisation:

\begin{theoremothers}\cite[Theorem 5.5]{GJ}
  \label{GJ}
  Let $S$ be a precompact translation surface, and let $\Gamma(S)$ be its
  Veech group. The following statements are equivalent.
  \begin{enumerate}
  \item 
    The groups $\Gamma(S)$ and $\rm SL(2,\Z)$ are commensurable.
  \item 
    Every cross ratio of saddle connections is rational. Equivalently 
    the field $\Kcr(S)$ is equal to $\Q$.
  \item 
    There exists a translation covering from a puncturing of 
    $S$ to a once-punctured flat torus.
  \item 
    $S$ is an origami up to an affine homeomorphism, \emph{i.e.} there is a 
    Euclidean parallelogram that tiles $S$ by translations.
  \end{enumerate}
\end{theoremothers}

The first result of this article explores what remains of the preceding characterisation 
if  $S$ is a general 
\emph{tame} translation surface.
 Tame translation surfaces 
are the translation surfaces all of whose singularities
are  cone angle singularities (possibly of infinite angle).  
This includes surfaces like $\R^{2}$, but also surfaces whose fundamental group 
is not finitely generated.
We define  tameness and the different types of singularities in \Cref{prelim}.
Furthermore, we call $S$  {\em maximal}, if it has no finite singularities
of total angle $2\pi$; compare \Cref{def-origamis}.

\begin{theorem}\label{mainthm}
  Let $S$ be a maximal tame translation surface. Then,
  \begin{enumerate}
  \item\label{Thm1i} 
    $S$ is affine equivalent to an origami if and only if the set 
    of \emph{developed cone points} 
    is contained in $L+x$, where $L\subset\R^2$ is a 
    lattice and $x\in\R^2$ is fixed.
  \item\label{partB} 
    If $S$ is an origami the following 
    statements (\labelcref{Kcr})-(\labelcref{Ksc}) hold. 
    In (\labelcref{vg}) and (\labelcref{fKtr}) we require in 
    addition that there are at least 
    two nonparallel saddle connections on $S$:
    \begin{enumerate}\renewcommand{\theenumi}{}
    \item\label{vg} 
      The Veech group of $S$ is commensurable to a subgroup of $\rm SL(2,\Z)$.
    \item\label{Kcr} 
      The field of cross ratios $\emKcr(S)$ is isomorphic to $\Q$.
    \item\label{Khol}
      The holonomy field $\emKhol(S)$ is isomorphic to $\Q$.
    \item\label{Ksc}
      The saddle connection field $\emKsc(S)$ is isomorphic to $\Q$.
    \item\label{Ktr}\label{fKtr}
      The trace field $\emKtr(S)$ is isomorphic to $\Q$.
    \end{enumerate}
    However, none of (\labelcref{vg}) - (\labelcref{fKtr}) implies 
    that $S$ is an origami.
  \end{enumerate}
\end{theorem}

In the proof of \Cref{mainthm} we will show that even if we require 
that in (\labelcref{vg}) the Veech group of $S$ is equal to $\rm SL(2,\Z)$, 
this condition does not imply that $S$ is an origami. \\

If $S$ is precompact, then the four fields $\Ktr(S)$, $\Khol(S)$,
$\Kcr(S)$ and $\Ksc(S)$ are number fields and we have 
the following hierarchy:
\begin{equation}
  \label{E:hier}
  \Q\subseteq\Ktr(S) \subseteq \Khol(S) \subseteq \Kcr(S) = \Ksc(S)
\end{equation}
Thus by Theorem~\Cref{GJ} the conditions (\labelcref{vg}), 
(\labelcref{Kcr}) and (\labelcref{Ksc})
in  (\labelcref{partB}) of \Cref{mainthm} are, for precompact surfaces,
equivalent to being an origami. Conditions (\labelcref{Khol}) and (\labelcref{Ktr}), 
however, are even for precompact translation surfaces
not equivalent to being an origami. Indeed, recall that
``the general'' precompact translation surface has trivial Veech group, 
\emph{i.e.} Veech group $\{I,-I\}$,
where $I$ is the identity matrix (see \cite[Thm.2.1]{M}). This implies that 
(\labelcref{Ktr}) is not equivalent to being an origami. 
Furthermore, in \Cref{E:5} 
we construct an explicit example of a precompact surface $S$ that is not an origami 
and such that $\Khol(S)=\Q$. This shows that (\labelcref{Khol})  
is not equivalent to being an origami.\\

In the case of general tame translation surfaces, 
the fields  $\Ktr(S)$, $\Khol(S)$,
$\Kcr(S)$ and $\Ksc(S)$ are not necessarily number fields anymore; compare \Cref{alg}.
Furthermore from the hierarchy in \labelcref{E:hier} it just remains true in general
that $\Khol(S)$ and $\Kcr(S)$ are both subfields of $\Ksc(S)$.
Some of the other relations in \labelcref{E:hier} hold under
extra assumptions on $S$; compare
\Cref{corrolary-relations-between-fields}.
It follows that, in general, if $\Ksc(S)$ is isomorphic to $\Q$, then both 
$\Khol(S)$ and $\Kcr(S)$ 
are isomorphic to $\Q$. In terms of \Cref{mainthm}, part (\labelcref{partB}), 
this is equivalent to say that (\labelcref{Ksc}) implies both 
(\labelcref{Kcr}) and (\labelcref{Khol}). 
Note that furthermore trivially  (\labelcref{vg}) implies (\labelcref{Ktr}). 
We treat the remaining of these implications in the next theorem.\\  

\begin{theorem}\label{rel}
  There are examples of tame translation surfaces $S$ for which
  \begin{enumerate}
  \item\label{Thm2i}
    The Veech group $\Gamma(S)$ is equal to ${\rm SL(2,\Z)}$ 
    and $K$ is not equal to $\Q$, where
    $K$ can be chosen from $\emKcr(S)$,  $\emKhol(S)$ and $\emKsc(S)$.
  \item\label{Thm2ii}  
    The fields $\emKsc(S)$ (hence also $\emKcr(S)$ and 
    $\emKhol(S)$) and $\emKtr(S)$ are equal to $\Q$, but $\Gamma(S)$ 
    is not commensurable to a subgroup of $\rm SL(2,\Z)$.
  \item\label{Thm2Part3} 
    $\emKcr(S)$ or $\emKhol(S)$ is equal to $\Q$, but $\emKsc(S)$ is not.
  \item\label{Thm2iv} 
    The field $\emKcr(S)$ is equal to $\Q$, but $\emKhol(S)$ is not or {\emph{vice versa}}: 
    $\emKhol(S)$ is equal to $\Q$, but $\emKcr(S)$ is not.
  \item\label{Thm2v}
    The field $\emKtr(S)$ is equal to $\Q$, but none of the conditions 
    (\labelcref{vg}), (\labelcref{Kcr}), (\labelcref{Khol}) or (\labelcref{Ksc}) 
    in \Cref{mainthm} hold. Moreover, none of the conditions (b), (c) or (d) imply that $\emKtr(S)$ is isomorphic to $\Q$. 
  \end{enumerate}
\end{theorem}

The proofs of the preceding two theorems heavily rely 
on modifications of the construction 
in \cite[Construction 4.9]{PSV} which was there used to determine all possible Veech groups
of tame translation surfaces. We summarise this construction in \Cref{subs:PSV}.
One can furthermore modify the construction to prove that any subgroup of $\rm SL(2,\Z)$
is the Veech group of an origami. From this we will deduce the 
following statement about the oriented outer automorphism 
group $\rm Out^+(F_2)$ of the free group $\rm F_2$ in two generators:
\begin{corollary}
\label{OUTF2}
Every subgroup of $\rm Out^+(F_2)$ is the stabiliser of a conjugacy class of some (possibly
infinite index) subgroup of $\rm F_2$.
\end{corollary}
\indent 
If $S$ is a precompact translation surface, the existence of hyperbolic 
elements, \emph{i.e.} matrices whose trace is bigger than 2, 
in $\Gamma(S)$ has consequences for the image 
of $H_1(S,\Z)$ in $\R^2$ under the developing map (also called \emph{holonomy} map; 
see \Cref{prelim}) 
and for the nature of some of the fields associated with $S$. To be more precise, if $S$ is precompact, the following is known: 

\begin{enumerate}\renewcommand{\theenumi}{\Alph{enumi}}\renewcommand{\labelenumi}{(\theenumi)}
 \item\label{Thm3Part1}
   If there exists $M\in\Gamma(S)$ hyperbolic, 
   then the holonomy field of $S$ is equal to $\Q[tr(M)]$. 
   In particular, the traces of any two hyperbolic 
   elements in $\Gamma(S)$ generate the same field over $\Q$; 
   see \cite[Theorem~28]{KS}.
 \item\label{Thm3Part2}  
   If there exists $M\in\Gamma(S)$ hyperbolic and 
   $tr(M)\in\Q$, then $S$ is an origami; see \cite[Theorem~9.8]{Mc1}.
 \item\label{Thm3Part3}
   If $S$ is a ``bouillabaisse surface" 
   (\emph{i.e.} if $\Gamma(S)$ contains 
   two transverse parabolic elements), then $\Ktr(S)$ is 
   totally real; compare \cite[ Theorem 1.1]{HL}. 
   This implies that if there exists an hyperbolic $M$ 
   in $\Gamma(S)$ such that $\Q[tr(M)]$ is not totally 
   real then $\Gamma(S)$ does not contain any parabolic elements; 
   see Theorem~1.2 in \emph{ibid.}
 \item\label{Thm3Part4}
   Let $\Lambda$ and $\Lambda_0$ be the subgroups of $\R^2$ generated 
   by the image under the holonomy map of  $H_1(\compl{S},\ZZ)$ and  
   $H_1(\compl{S},\Sigma ; \ZZ)$, respectively. Here $\Sigma$ is the set 
   of cone angle singularities of $S$. If the affine group of $S$ contains a pseudo-Anosov element, 
   then $\Lambda$ has finite index in $\Lambda_0$; see \cite[Theorem~30]{KS}.
\end{enumerate}
\indent The third main result of this paper shows that 
when passing to general tame translation surfaces 
there are no such consequences. For such surfaces, an element of 
$\Gamma(S)<{\rm GL_+(2,\R)}$ will be called hyperbolic, 
parabolic or elliptic if its image in $\rm PSL(2,\R)$
is  hyperbolic, parabolic or elliptic respectively.
\begin{theorem}
  \label{ncon}
There are examples of tame translation surfaces $S$ for which 
(\labelcref{Thm3Part1}), (\labelcref{Thm3Part2}), (\labelcref{Thm3Part3}) 
or (\labelcref{Thm3Part4}) from above do not hold.
\end{theorem}
\indent 
We remark that all tame translation surfaces $S$ that we construct 
in the proof of the preceding theorem have the same topological type: 
one end and infinite genus. Such surfaces are called Loch Ness monster; see 
\Cref{prelim}.\\

\indent This paper is organised as follows. In \Cref{prelim} 
we review the basics about general translation surfaces, tame translation surfaces
and origamis, 
their singularities and possible Veech groups. 
In \Cref{section3} we present the definitions of the fields listed in \Cref{mainthm} 
for general tame translation surfaces. We prove that the main 
algebraic properties of these fields which are true for precompact translation
surfaces no longer hold for general translation surfaces. 
For example, we construct examples of tame translation surfaces 
for which the trace field is not a number field. We furthermore show those
inclusions from \labelcref{E:hier} which still are valid for tame 
translation surfaces.
Section 4 deals with the proofs of the three theorems stated in this section.
 We refer the reader to \cite{HS}, \cite{HLT} or \cite{HW} for recent developments concerning 
tame translation surfaces.\\\\
\\
\textbf{Acknowledgements}. 
Both authors would like to express their gratitude
to the Hausdorff Research Institute for Mathematics for their hospitality and 
wonderful working environment within the Trimester program 
{\em Geometry and dynamics of Teichm\"uller space}.  The first author would like to thank the organisers of the conference \emph{Algebraic Methods in Geometry: Commutative and Homological Algebra in Foliations and Singularities}, in honour of Xavier G\'omez-Mont on the occasion of his 60th birthday, where part of the results of this article were presented.
The second author would 
furthermore like to thank Frank Herrlich, who has proofread a previous version.
Both authors are thankful to the referees who have substantially helped to improve the 
readability.

\section{Preliminaries}
  \label{prelim}

\subsection{General translation surfaces and their singularities}
\indent In this section we review some basic notions needed for the rest of the article. For a detailed exposition, we refer to \cite{GJ} and \cite{Thu}.\\
\\
\indent A \emph{translation surface} $S$ will be a 2-dimensional real $G$-manifold with $G=\R^2=Trans(\R^2)$; that is, a surface on which coordinate changes are translations of the real plane $\R^2$. We can pull back to $S$ the standard translation invariant flat metric of the plane  and obtain this way a flat metric on the surface. 
We denote by $\overline{S}$ the metric completion of $S$ with respect 
to this natural flat metric. A \emph{translation map} is  a $G$-map between 
translation surfaces. Every translation map 
$f:S_1\longrightarrow S_2$ has a unique continuous 
extension $\overline{f}:\overline{S_1}\longrightarrow\overline{S_2}$.
\begin{definition}\label{def-general-translation-surface}
If $\compl{S}$ is homeomorphic to an orientable compact surface, we say that $S$ is a {\em precompact translation surface}. Else we say that $S$ is not precompact. Observe that a not precompact translation surface is not necessarily of infinite type. The union of all precompact and not precompact translation surfaces form the set of \emph{general translation surfaces}. 
\end{definition}
\begin{definition}
Let $S$ be a translation surface. We call the points of $\overline{S}\backslash S$ \emph{singularities}
of the translation surface $S$. A point $x\in\compl S \setminus S$ is called a \emph{finite angle singularity} or \emph{finite angle cone point} 
of total angle $2\pi m$, where $m\geq 1$ is a natural number, if there exists a neighbourhood of $x$ which is isometric to a neighbourhood of the origin in $\R^{2}$ with a metric that, in polar coordinates $(r,\theta)$, has the form $ds^{2}=dr^{2}+(mrd\theta)$. The set of finite angle singularities of $\compl{S}$ is denoted by $\Sigmafin$.
\end{definition}

Precompact translation surfaces are obtained 
by glueing finitely many polygons (deprived of their vertices)
along parallel edges by translations. One even obtains
all precompact translation surfaces in this way; see \cite{Masur}.
Thus if $S$ is a precompact translation surface, all of its 
singularities are finite angle singularities. If furthermore 
$\compl{S}$ has genus at least $2$, then, by a simple Euler characteristic
calculation, $\compl{S}$  always has singularities. For non precompact translation surfaces, new kinds of singularities will occur. 
We illustrate this in the following example.

\begin{example}\label{staircase}
In \Cref{figure-staircase} we depict a translation surface
obtained from infinitely many copies of the Euclidean unit square.
More precisely, we remove the vertices from all the squares 
in the figure. Some pairs of edges are already identified; among the remaining edges
we  identify opposite ones which are labelled by the same letter 
by translations. The result is a translation surface $S$ which is not precompact. 
It is called \emph{infinite staircase} because of its shape. This and similar shaped
surfaces have been intensively studied in the literature; see e.g. \cite{HS},
\cite{HW}, \cite{HW1} and \cite{CG}.
$S$ is a prototype for what we 
will call in this text an infinite origami or infinite square-tiled surface; compare
\Cref{def-origamis}. The translation surface $S$ comes with a natural cover $p$ 
to the once punctured torus obtained
from glueing parallel edges of the Euclidean unit square again with its vertices
removed.\\ 
Observe furthermore that the metric completion of the infinite staircase 
$S$ has four singularities $x_{1}$, $x_{2}$, $x_{3}$ and $x_{4}$. 
Restricted to  a punctured neighbourhood of them $p$ is infinite cyclic and 
the universal cover of a once punctured disk.  In this sense
the singularities $x_1$, \ldots, $x_4$ generalise 
finite angle singularities of angle $2\pi m$. They are prototypes for
what we call infinite angle singularities; compare \Cref{def-infinite-angle-singularity}. 
\end{example}
\begin{figure}
\begin{center}
\includegraphics[scale=0.8]{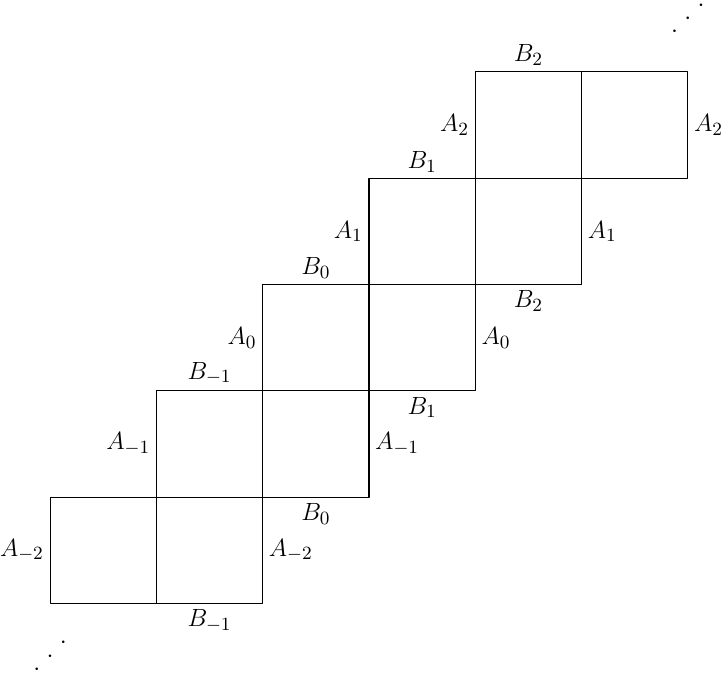}\\
\caption{An infinite-type translation surface}\label{figure-staircase}
\end{center}
\end{figure}
 
\begin{definition}\label{def-infinite-angle-singularity}
Let $S$ be a translation surface. A point $x\in\compl{S}$ is called an 
\emph{infinite angle singularity} or \emph{infinite angle cone point} 
if there exists a neighbourhood of $x$ isometric to the neighbourhood of the branching point of the infinite cyclic flat branched covering of $\R^{2}$.  The set of infinite angle singularities of $\compl{S}$ is denoted by $\Sigmainf$. Points in the set $\Sigma=\Sigmafin\cup\Sigmainf$ will be called \emph{cone angle singularities} of $S$ or just \emph{cone points}.
\end{definition}

\begin{definition}
A translation surface $S$ is called \emph{tame} if all points in $\compl{S}\setminus S$ 
are cone angle singularities (of finite or infinite total angle). A 
tame translation surface $S$ is said to be of \emph{infinite-type}
if the fundamental group of $S$ is not finitely generated. 
\end{definition}
Every precompact translation surface is tame. There are tame translation surfaces with infinite angle singularities which are not of infinite type. For example 
consider the infinite cyclic covering of the once punctured plane.
 Explicit examples of infinite-type translation surfaces arise naturally when studying the billiard game on a polygon whose interior angles are not rational multiples of $\pi$ (see \cite{V1}). 
Nevertheless, not all translation surfaces are tame. 
If one allows 
infinitely many polygons, {\em wild} types of singularities
may occur. Simple examples of not tame translation surfaces can be found in \cite{C} and \cite{BV}.
\\

In the following we define the very special class of translation surfaces
called origamis or square-tiled surfaces. 
With \Cref{staircase} we have already seen a specific example of them.

\begin{definition}\label{def-origamis}
  A translation surface is called {\em origami} or 
  {\em square-tiled surface}, if it fulfils one of the two following 
  equivalent conditions:
  \begin{enumerate}
  \item 
    $S$ is a translation surface obtained from glueing (possibly infinitely many) copies
    of the Euclidean square along edges by translations
    according to the following rules:
    \begin{itemize}
    \item
      each left edge is glued to precisely one  right edge,
    \item
      each upper edge to precisely one lower edge and 
    \item
      the resulting surface is connected;
    \end{itemize}
      and removing all singularities.
  \item 
    $S$ allows an unramified covering $p:S^* \to T_0$ of the once-punctured unit torus $T_0 = 
    (\RR^2\backslash L_0)/L_0$, such that $p$ is a translation map. Here 
    $S^*$ is a subset of $S$ such that the complement $S\backslash S^*$
    is a discrete set of points on $S$.
    $L_0$ is the lattice in $\R^2$ spanned by the two standard basis vectors 
    $e_1 = (1,0)$ and $e_2 = (0,1)$. Furthermore, $S$ is maximal in the sense that
    $\overline{S}\backslash S$ contains no finite angle singularities of angle $2\pi$.
  \end{enumerate}
  An origami  will be called \emph{finite} if the number of 
  squares needed to construct it is finite or, equivalently, 
  if the unramified covering  $p:S \to T_0$ is finite. 
  Else, the origami will be called \emph{infinite}. 
  See \cite[Section 1]{Snonc} for a detailed introduction to finite origamis. Infinite 
  origamis were studied e.g. in \cite{HS} and \cite{G}.
\end{definition}


\subsection{Developed cone points and the Veech group}

\indent 
In the following we introduce the \emph{set of developed cone points}
for tame translation surfaces,
which will 
play an important role in the proof of \Cref{mainthm}.
Let $\pi_S:\widetilde{S}\longrightarrow S$ be a universal cover of a translation surface $S$ and ${\rm Aut}(\pi_S)$
the group of its deck transformations. From now on, $\widetilde{S}$ 
is endowed with the translation structure obtained as pull-back from the one on $S$ via $\pi_S$. 
Recall from \cite[Section 4.3]{Thu} that for every deck transformation $\gamma$, 
there is a unique translation hol$(\gamma)$ satisfying
\begin{equation}
  \label{hol}
  \rm
  dev\circ\gamma=hol(\gamma)\circ dev,
\end{equation}
where $\rm dev:\widetilde{S}\longrightarrow \mathbf{R}^2$ denotes the developing map.
The map hol$: {\rm Aut}(\pi_S) \to \mbox{Trans}(\R^2) \cong \RR^2$ is a group homomorphism.
By considering the continuous extension of each map in \Cref{hol} to the metric completion of its domain, we obtain
\begin{equation}
  \label{ehol}
  \rm
  \overline{dev}\circ\overline{\gamma}=hol(\gamma)\circ \overline{dev}.
\end{equation}
Overall we have the following commutative diagram:
\begin{equation}
  \xymatrix{
         && \overline{\widetilde{S}} \ar[rrrr]^{\overline{\gamma}} \ar@/_/[dddll]_{\overline{\rm dev}}&
            && & \overline{\widetilde{S}} \ar@/^/[dddrr]^{\overline{\rm dev}}&& \\
         && & \widetilde{S} \ar@{_(->}[ul]\ar[rr]^{\gamma} \ar[d]_{\pi_S} \ar[ddlll]_{\rm dev}
            && \widetilde{S} \ar@{^(->}[ur]   \ar[d]^{\pi_S} \ar[ddrrr]^{\rm dev}& &&\\
         && & S             &&     S         & &&\\
    \R^2 \ar[rrrrrrrr]^{\rm hol(\gamma)}&& & && & && \R^2
    \hspace*{5mm}.
    }
    \end{equation}
\begin{definition}
 \emph{The set of developed singularities} of $S$ is  the 
subset of the plane $\R^2$ given by 
$\overline{\rm dev}(\overline{\widetilde{S}}\setminus \widetilde{S})$. 
We denote it by $\widetilde{\Sigma}(S)$. If $S$ is a tame
translation surface, we also call $\widetilde{\Sigma}(S)$ the 
\emph{set of developed cone points}.
\end{definition}

\begin{definition}
  \label{singgeo}
  A \emph{singular geodesic} of a translation surface $S$ is an open geodesic segment 
  in the flat metric of $S$ whose image under the natural 
  embedding $S\hookrightarrow\compl{S}$ issues from a singularity 
  of $\compl{S}$, contains no singularity in its interior 
  and is not properly contained in some other geodesic segment. 
  A \emph{saddle connection} is a finite length singular geodesic.
\end{definition}
\indent To each saddle connection we can associate a \emph{holonomy vector}: we 'develop' the saddle connection in the plane by using local coordinates of the flat structure. The difference vector defined by the planar line segment is the holonomy vector. Two saddle connections are \emph{parallel}, if their corresponding holonomy vectors are linearly dependent. \\
\\
\indent Next, we introduce the {\em Veech group}, which 
since Veech's article \cite{Ve} from 1989 has been
studied for precompact translation surfaces 
as the natural object associated with the surface. 
Let $\mathrm{Aff}_+(S)$ be the group of affine orientation 
preserving homeomorphisms of a translation surface $S$. 
Consider  the map \begin{equation}
    \label{eseq}
     \mathrm{Aff}_+(S)\overset{D}\longrightarrow \mathbf{GL}_+(2,\R)
\end{equation}
 that associates to every $\phi\in \mathrm{Aff}_+(S)$ its (constant) Jacobian derivative $D\phi$.
 \begin{definition}
    \label{vgroup}
    Let $S$ be a translation surface. We call $\Gamma(S)=D(\mathrm{Aff}_+(S))$ the \emph{Veech group} of $S$.
 \end{definition}

\begin{remark}\label{remark-action}
The group $\mathbf{GL}_{+}(2,\R)$ naturally acts on the set of translation surfaces: We define
$A\cdot S$ to be the translation surface obtained from $S$ by composing each chart in the translation
atlas with the linear map ${x\choose y} \mapsto A\cdot {x\choose y}$. Since the map $\mathrm{id}_A: S \to A\cdot S$
which topologically is the identity map has derivative $A$, we have
that $\Gamma(A\cdot S) = A\cdot\Gamma(S)\cdot A^{-1}$.
\end{remark}

\subsection{Constructing tame surfaces with prescribed Veech groups.} 
  \label{subs:PSV}
The proofs of our main results heavily rely on slight modifications of 
the construction in the proof of \cite[Proposition 4.1]{PSV}. 
In this section we review this construction. 
We will mainly use the notation of \cite{PSV}. \\
The construction we are about to review proves the following:
\begin{proposition}[{\cite[Proposition 4.1]{PSV}}]
  \label{P:PSV}
  For any countable subgroup $G$ of ${\rm GL_{+}(2,\R)}$ disjoint from 
  $\{\mathcal{U}=g\in \rm \mathbf{GL}_{+}(2,\R)\hspace{1mm}:\hspace{1mm} || g||<1\}$ 
  there exists a tame translation surface $S=S(G)$, which is homeomorphic to the 
  \emph{Loch Ness monster}, with Veech group $G$.
\end{proposition}
The {\em Loch Ness monster} is the unique surface $S$ (up to homeomorphism) of infinite genus and one end. By one end we mean that for every compact set $K\subset S$ there exists a compact set $K\subset K'\subset S$ such that $S\setminus K'$ is connected. We refer the reader to \cite{R} for a more detailed discussion on surfaces of infinite genus and ends. \\

First we have to recall a basic geometric operation which will play an important
role in the construction: \emph{glueing translation surfaces along marks}.

\begin{definition}
Let $S$ be a tame translation surface. A \emph{mark} on $S$ is an oriented finite length geodesic (with endpoints) on $S$. The \emph{vector} of a mark is its holonomy vector, which lies in $\R^2$. If $m, m'$ are two disjoint marks on $S$ with equal vectors, we can perform the following operation. We cut $S$ along $m$ and $m'$, which turns $S$ into a surface with boundary consisting of four straight segments.  Then we reglue these segments to obtain a tame translation surface $S'$ different from the one we started from. We say that $S'$ is obtained from $S$ by \emph{reglueing along $m$ and $m'$}. Let $S_0=S\setminus (m\cup m')$. Then $S'$ admits a natural embedding $i$ of $S_0$. If $A\subset S_0$, then we say that $i(A)$ is \emph{inherited} by $S'$ from $A$. 
\end{definition}

\begin{remark}
\label{Euler and 4pi}
If $S'$ is obtained from $S$ by reglueing, then
the number of singularities of $S'$ of a fixed angle equals the one of $S$, except for $4\pi$--angle singularities, whose number in $S'$ 
is greater by $2$ to that in $S$ (we put $\infty+2=\infty$). The Euler characteristic of $S$ is greater by $2$ than the Euler characteristic of $S'$.\\
We can extend the notion of reglueing to ordered families $\mathcal{M}=(m_n)_{n=1}^{\infty}$ and $\mathcal{M'}=(m'_n)_{n=1}^{\infty}$ of disjoint marks, which do not accumulate in $\overline{S}$, and such that the vector of $m_n$ equals the vector of $m'_n$, for each $n$.
\end{remark}
\emph{Outline of the construction}. Let $\{a_{i}\}_{i\in I}$ (with $I\subseteq \NN$)
be a (possibly infinite) set of generators for $G$. We make use of the fact that any group $G$ acts on its Cayley graph $\Gamma$ and turn the graph $\Gamma$ in a $G$-equivariant way into a translation surface. In the following we describe the general idea of the construction; 
below we give the explicit construction for the case that $G$ is generated by
two elements. The construction then works just in the same way
for general groups; compare \cite[Construction 4.9]{PSV}.
\begin{itemize}
\item
  With each vertex $g$ of $\Gamma$ we associate a translation surface $V_g$.
  More precisely we start from some translation surface $V_{\I}$ and
  define $V_g$ to be its translate $g\cdot V_{\I}$ by the action of $\mathbf{GL}_{+}(2,\R)$
  on the set of translation surfaces described in \Cref{remark-action}.
  Observe that the linear group $G$ naturally acts via affine homeomorphisms
  on the disjoint union of the $V_g$'s; an element $h \in G$ maps $V_g$
  to $V_{h\cdot g}$.
  In the next step we will choose disjoint marks on the translation surfaces $V_g$.
  Reglueing the disjoint union of the surfaces $V_g$ along these marks will give
  us a connected surface on which $G$ acts by affine homeomorphisms.\\
  At the moment, we can assume $V_{\I}$ just to be the real plane $\RR^2$
  equipped with an origin and a coordinate system.
\item
  We choose marks on the starting surface $V_{\I}$ in the following way:
  \begin{itemize}
  \item
    For each $i$ in $I$  we choose a family $\CCC^i = \{m^i_j\}_{j \in J}$ 
    (with $J \subseteq \NN$) 
    of horizontal marks $m^i_j$  
    of length 1,
    \emph{i.e.} the vector of each mark $m^i_j$ is the first standard basis vector $e_1$. 
  \item
    For each $i$ we choose a family $\CCC^{-i} = \{m^{-i}_j\}_{j \in J}$ of marks
    with vector $a_i^{-1}(e_1)$, \emph{i.e.} the vector of $m^{-i}_j$ 
    is equal to $a_i^{-1}\cdot e_1$.
  \item
    All marks are disjoint.
  \end{itemize}
\item
  On each $V_g$ we take the corresponding marks
  $g(m^i_j)$ and $g(m^{-i}_j)$ with $i \in I$ and $j \in \NN$.
  The mark $g(m^i_j)$ has the vector $g\cdot e_1$ and $g(m^{-i}_j)$
  has the vector $g a_i^{-1}\cdot e_1$.
\item
  We pair the mark $g(m^i_j)$ on the surface $V_g$
  with the mark $ga_i(m^{-i}_j)$ on $V_{g a_i}$.
  Observe that for both the vector is $g\cdot e_1$.\\
  We now reglue the disjoint union of the $V_g$'s along these
  pairs of marks.
\end{itemize} 
This gives us a translation surface $S_1$ on which  the elements of $G$
act via affine homeomorphisms, \emph{i.e.} $\Gamma(S_1)$ contains
$G$. However we are not yet done, but still have the following problems:
\begin{enumerate}
\item 
  The Veech group $\Gamma(S_1)$ can be bigger than $G$.
\item
  The singularities can accumulate. In this case $S_1$ is not tame. 
\item
  We want the translation surface to have  one end.
\end{enumerate}
We resolve the problems in the following way: To 
enforce that all elements in the Veech group are in $G$, we will modify 
the starting surface $V_{\I}$.
We will replace it by a surface obtained from glueing 
a \emph{decorated surface} $\widetilde{L}'_{\I}$ (described below) to a
plane $A_{\I} = \RR^2$. The surface $\widetilde{L}'_{\I}$ will be 
decorated with special singularities. This will guarantee that every 
orientation preserving affine homeomorphism 
permutes the set of the singularities on the $\widetilde{L}'_{g}$'s  
and with some more care we will  
establish that it actually acts as one of the elements of $G$.
To avoid accumulation of singularities, we will associate with each edge 
in the Cayley graph  between two vertices $g$ and $g'$ (let us say that $g^{-1}g' = a_i$
is the $i$'th generator) 
a \emph{buffer surface} $\hat{E}^i_g$ which connects $V_g$ to $V_{g'}$, 
but separates them by a definite distance. 
Finally, we keep track of the end by providing that each $V_g$ and 
$\hat{E}^i_g$ is one-ended and that after glueing all $V_g$ and $\hat{E}^i_g$, 
their ends actually merge into one end. This actually is the reason
why we have to choose infinite families of marks. If we do not require
the surface to be a Loch Ness monster, then it suffices to take one mark
from each infinite family.\\

\emph{An illustrative example}. In the following paragraphs we carry 
out the construction for the case where $G$ is 
generated by two matrices $a_{1}$ and $a_{2}$. The general case works in the same
way; compare \cite[Construction 4.9]{PSV}.\\
{\bf Constructing the translation surface \boldmath{$V_{g}$}}: 
We first construct the surface $V_{Id}$. We will obtain it by glueing two surfaces $A_{Id}$ 
and $\widetilde{L}'_{Id}$ along an infinite family of marks. 
Let $A_{\I}$ be an oriented flat plane, equipped with an origin and the standard
basis $e_1 = (1,0)$ and $e_2 = (0,1)$. We define the families of marks as follows:
\begin{itemize}
\item
  For $i=0,1,2$ let $\mathcal{C}^i$ be the family of marks on $A_{\mathrm{Id}}$ with 
  endpoints $i\f+(2n-1)\e,\ i\f+2n\e$, for $n\geq 1$. 
  All these marks are pairwise disjoint.
\item
  Given $x_1, y_1\in \R$, consider the family $\mathcal{C}^{-1}$ of 
  marks on $A_{\mathrm{Id}}$ with endpoints $(nx_1,y_1),\ (nx_1, y_1)+a_1^{-1}(\e)$, 
  for $n\geq 1$. We can choose $x_1>0$ sufficiently large and $y_1<0$ sufficiently small 
  so that all these marks are pairwise 
  disjoint and disjoint from the ones in $\mathcal{C}^i$ for $i=0,1,2$. 
\item 
  Observe that a translate of the lower half-plane in $A_{\mathrm{Id}}$ 
  is avoided by all already constructed marks.
  In this way we can choose $x_2,-y_2\in \R$ sufficiently large so that the 
  marks with endpoints $(nx_2,y_2),\ (nx_2, y_2)+a_2^{-1}(\e)$, 
  for $n\geq 1$, are pairwise disjoint and disjoint with the previously 
  constructed marks. We denote this family by $\mathcal{C}^{-2}$.
\end{itemize}
Let $L_{\I}$ be an oriented flat plane, equipped with an origin $O_{\I}$. Let $\widetilde{L}_{\I}$ be the threefold cyclic branched covering of $L_{\I}$, which is branched over the origin.
Denote the projection map from $\widetilde{L}_\I$ onto $L_\I$ by $\pi$. Denote by $R$ the closure in $\widetilde{L}_{\mathrm{Id}}$ of one connected component of the preimage under $\pi$ of the open right half-plane in $L_{\mathrm{Id}}$. On $R$ consider coordinates induced from $L_{\mathrm{Id}}$ via $\pi$. We define the following family of marks on $\widetilde{L}_{\I}$:
\begin{itemize}
\item
  Let $\mathcal{C}'$ be the family of marks in $R$ 
  with endpoints $(2n-1)\e, 2n\e$, for $n\geq 1$.
\item 
  Let $t$ and $b$ be the two marks in $\widetilde{L}_{\mathrm{Id}}$ 
  with endpoints in $R$ with coordinates $\f,2\f$ and $-2\f, -\f$, 
  respectively. 
\end{itemize}
Let $\widetilde{L}'_{\I}$ be the tame flat 
surface obtained from $\widetilde{L}_{\I}$ by reglueing along $t$ and $b$. 
We call $\widetilde{L}'_{\I}$ the \emph{decorated surface}.
Finally, we obtain $V_{\rm Id}$ by glueing $A_{\rm Id}$ with $\widetilde{L}'_{\I}$ 
along the families of marks $\mathcal{C}^{0}$ and $\mathcal{C}'$.
For each $g\in G$ we define $V_{g}$ as the translation surface $g\cdot V_{\I}$.\\
Observe that if
we denote by $\widetilde{O}_g$ the unique preimage on $\widetilde{L}_{g}$ 
of the origin $O_{g}$ of $L_{g}$ via the three-fold covering, then $\widetilde{O}_g$
is a singularity of total angle $6\pi$ and there are precisely three 
saddle connections starting in $\widetilde{O}_g$.\\
{\bf Constructing the buffer surface \boldmath{$\hat{E}^i_g$}:} 
Let $E_{\I}, E'_{\I}$ be two oriented flat planes, equipped with origins that allow us to identify them with $\R^2$. We define the following families of vector $\e$ marks 
on $E_{\I}\cup E'_{\I}$. 
\begin{itemize}
\item
  Let $\mathcal{S}$ be the family of marks on $E_{\I}$ 
  with endpoints $4n\e, (4n+1)\e$, for $n\geq 1$.
\item
  Let $\mathcal{S}_\gl$ be the family of marks on $E_{\I}$ 
  with endpoints $(4n+2)\e, (4n+3)\e$, for $n\geq 1$. 
\item
  Let $\mathcal{S}'$ be the family of marks on $E'_{\I}$ 
  with endpoints $2n\f, 2n\f+\e$, for $n\geq 1$.
\item
  Finally, let $\mathcal{S}'_\gl$ be the family of marks 
  on $E'_{\I}$ with endpoints $(2n+1)\f,(2n+1)\f+\e$, for $n\geq 1$. 
\end{itemize}
Let $\hat{E}_{\I}$ be the tame flat surface obtained from 
$E_{\I}$ and $E'_{\I}$ by reglueing along $\mathcal{S}_\gl$ and $\mathcal{S}'_\gl$. We call $\hat{E}_{\I}$ the \emph{buffer surface}. 
The surface $\hat{E}_{\I}$ comes with the distinguished families of marks 
inherited from $\mathcal{S}$ and $\mathcal{S}'$, for which we retain
the same notation. Let $\hat{E}^1_{\I}$ and $\hat{E}^2_{\I}$
be two copies of $\hat{E}_{\I}$ and for  each $g\in G$ let $\hat{E}^i_g$
to be the translation surface $g\cdot \hat{E}^i_g$ ($i \in \{1,2\}$). It is endowed with
the two family of marks $\mathcal{S}^i_g$ and $\mathcal{S}'^i_g$.\\
{\bf Construction of the surface \boldmath{$S$}:}
We finally obtain the desired surface $S$ from the disjoint
union of all $V_g$'s and $\hat{E}^i_g$ in the following way: 
\begin{itemize}
\item
  Reglue each mark $\mathcal{C}^1_g$ on $V_g$ with $\mathcal{S}^1_g$ 
  on $\hat{E}^1_g$, 
  and each mark $\mathcal{S}'^1_g$ on $\hat{E}^1_g$ with  $\mathcal{C}^{-1}_{ga_1}$
  on $V_{ga_1}$.
\item
   Reglue each mark $\mathcal{C}^2_g$ on $V_g$ with $\mathcal{S}^2_g$ 
  on $\hat{E}^2_g$, 
  and each mark $\mathcal{S}'^2_g$ on $\hat{E}^2_g$ with  $\mathcal{C}^{-2}_{ga_2}$
  on $V_{ga_2}$. 
\end{itemize}
In \cite[Section 4]{PSV} it is carefully carried out that
the construction is well defined and gives the desired result 
from \Cref{P:PSV}.



\section{Fields associated with translation surfaces}
\label{section3}
There are four subfields of $\R$ in the literature 
which are naturally associated with
a translation surface $S$. They are called the {\em holonomy field} $\Khol(S)$, the {\em segment field} or 
{\em field of saddle connections} $\Ksc(S)$, the
{\em field of cross ratios of saddle connections} $\Kcr(S)$, and the {\em trace field} $\Ktr(S)$; compare \cite{KS} and \cite{GJ}. In the following, we extend
their definitions to (possibly
non precompact) tame translation surfaces.
\begin{remark}
  It follows from  \cite[Lemma 3.2]{PSV} that there 
  are only three types of tame translation surfaces such that $\overline{S}$ has  
  no singularity: $\R^2$, $\R^2/\Z$ and flat tori. Furthermore, tame translation 
  surfaces with only one 
  singularity are cyclic coverings of $\R^2$ ramified over the origin. 
  Finally, if $\overline{S}$ has at least two singularities,
  then there exists at least one saddle connection. 
\end{remark}
\begin{definition}\label{definition-fields}
  Let $S$ be a tame translation surface and $\compl{S}$
  the metric completion of $S$.
  \begin{enumerate}
  \item\label{def:KS} (Following \cite[Section 7]{KS}.) 
    Let $\Lambda$ be the image of $H_1(\compl{S},\ZZ)$ in $\RR^2$ under
    the holonomy map $h$  and let
    $n$ be the dimension of the smallest $\RR$-subspace of
    $\RR^2$ containing $\Lambda$; in particular $n$ is $0$, $1$ or $2$. 
    The {\em holonomy field {\em $\Khol(S)$}}
    is the smallest subfield $k$ of $\RR$ such 
    that \[\Lambda\otimes_{\ZZ} k\cong k^n.\]
  \item\label{def:Ksc} 
    Let $\Sigma$
    denote the set of all singularities of $\overline{S}$. 
    Using in (\labelcref{def:KS})  $H_1(\compl{S}, \Sigmafin; \ZZ)$, 
    the homology relative to the set of finite angle 
    singularities,
    instead of the absolute homology $H_1(\compl{S},\ZZ)$, 
    we obtain the {\em segment field} or {\em field of
    saddle connections {\em$\Ksc(S)$}.}
  \item\label{def:Kcr} (Following \cite[Section 5]{GJ}.) 
    The {\em field of cross ratios of
    saddle connections} $\Kcr(S)$ is the field generated by 
    the set of all cross ratios $(v_1, v_2;v_3,v_4)$, 
    where the $v_i$'s are four pairwise nonparallel holonomy 
    vectors of saddle connections of $S$; compare \Cref{remark3.3} iii). 
  \item\label{def:Ktr} 
    Finally, the {\em trace field} $\Ktr(S)$ is the
    field generated by the traces of elements in the Veech group:
    $\Ktr(S) = \QQ[\mbox{tr}(A)| A \in \Gamma(S)]$.
  \end{enumerate}
\end{definition}

In the rest of this section we mean by a {\em holonomy vector} always the holonomy vector 
of a saddle connection.\\

\begin{remark}\label{remark3.3}
  \begin{enumerate}
  \item\label{rem3.3i} 
    \Cref{definition-fields} (\labelcref{def:KS}) is equivalent to the following: If $n=2$, 
    take any two nonparallel vectors $\{e_1,e_2\}\subset\Lambda$, 
    then $\Khol(S)$ is the smallest subfield $k$ of $\RR$  
    such that every element $v$ of $\Lambda$
    can be written in the form $a\cdot e_1 + b\cdot e_2$, with $a,b \in k$. 
    If $n=1$, any element $v$ of $\Lambda$ can be written as $a\cdot e_1$, 
    with $a\in\Khol(S)$ and $e_1$ any nonzero (fixed) vector in $\Lambda$. 
    If $n=0$,  $\Khol(S) = \QQ$.\\
    The same is true for $\Ksc(S)$, if $\Lambda$ is the image of 
    $H_1(\compl{S}, \Sigmafin; \ZZ)$ in $\RR^2$. 
  \item
    Recall that $\compl{S}$ is a topological surface
    if and only if all of its singularities have finite cone angles.
    However, if $\Sigmainf$ (resp. $\Sigmafin$) is the set of infinite 
    (resp. finite) angle singularities, then
    $\semicompl{S} = \compl{S}\backslash\Sigmainf = S \cup \Sigmafin$
    is a surface, possibly of infinite genus. 
    We furthermore have that the fundamental group
    $\pi_1(\compl{S})$ equals $\pi_1(\semicompl{S})$ and thus 
    $H_1(\compl{S},\ZZ) \cong $ $H_1(\semicompl{S},\ZZ)$: 
    Indeed, for every infinite 
    angle singularity $p_0\in\overline{S}$, there exists by definition
    a neighbourhood $U$ of $p_0$ in $\overline{S}$ which is isometric
    to a neighbourhood of the branching point $z_0$ of the infinite flat cyclic covering
    $X_0$ of $\RR^2$ branched over $0$. Without loss of 
    generality we may choose the neighbourhood of $z_0$ as
    an open ball of radius $\varepsilon$ 
    in $X_0$. We then have that $U$ is homeomorphic to 
    $\{(x,y) \in \RR^2| x > 0\} \cup \{(0,0)\} \subset \R^2$.
    In particular, $U$ and $U\backslash\{p_0\}$ are both contractible, 
    and by the Seifert-van Kampen
    theorem we have $\pi_1(\compl{S}\backslash\{p_0\}) \cong \pi_1(\compl{S})$.
  \item
    Recall that the cross ratio $r$ of four vectors $v_1$, \ldots, $v_4$ with
    $v_i = (x_i,y_i)$ is equal to 
    the cross ratio of the real numbers $r_1 = y_1/x_1$, \ldots, $r_4 = y_4/x_4$,
    \emph{i.e.} 
    \begin{equation}
      \label{CR}
      (v_1,v_2;v_3,v_4) = \frac{(r_1 - r_3)\cdot(r_2-r_4)}{(r_2-r_3)\cdot(r_1 - r_4)}.
    \end{equation}
    If $r_i=\infty$ for some $i=1,\ldots,4$, one eliminates 
    the factors on which it appears in \Cref{CR}. 
    For example, if $r_1=\infty$, then $(v_1,v_2;v_3,v_4)=\frac{r_2-r_4}{r_2-r_3}$. 
    If there are no four non parallel holonomy vectors,  $\Kcr(S)$ is equal to $\Q$.
  \item\label{rem3.3iv}
    The four fields from \Cref{definition-fields} are invariant under the
    action of $\mathbf{GL}(2,\R)$ described in \Cref{remark-action}, {\em i.e.} we have
    for $A \in \mathbf{GL}(2,\R)$
    \[
    \begin{array}{lcllcl}
      \Khol(S)&=& \Khol(A\cdot S), &\Ksc(S) &=& \Ksc(A\cdot S),\\ 
      \Kcr(S)&=& \Kcr(A\cdot S),   &\Ktr(S) &=& \Ktr(A\cdot S).
    \end{array}
    \]
    For $\Khol(S)$ and $\Ksc(S)$ this follows from (\labelcref{rem3.3i}). 
    Recall that the cross ratio is invariant under linear transformation. Thus
    the claim is true for the field $\Kcr(S)$. Finally, we have that 
    $\Gamma(A\cdot S)$ is conjugated to $\Gamma(S)$; compare \Cref{remark-action}.
    Since the trace of a matrix is invariant under conjugation, the claim
    also holds for $\Ktr(S)$.
  \end{enumerate}
\end{remark}

It follows directly from the definitions that $\Khol(S) \subseteq \Ksc(S)$.
Furthermore, we see from \Cref{remark3.3}
that $\Kcr(S) \subseteq \Ksc(S)$: Suppose $S$ has two linearly independent
holonomy vectors $w_1$ and $w_2$. By (\labelcref{rem3.3iv}) in the preceding remark we may assume that $w_1 = e_1$,
$w_2 = e_2$ is the standard basis. Let $v_1$, $v_2$, $v_3$, $v_4$ be four 
arbitrary pairwise 
nonparallel holonomy vectors with $v_i = (x_i,y_i)$.
By (\labelcref{rem3.3i}) we have that all the coordinates $x_i$ and $y_i$
are in $\Ksc(S)$. Thus in particular the cross ratio
$(v_1,v_2;v_3,v_4)$ is in $\Ksc(S)$. If there is no pair $(w_1,w_2)$ 
of linearly
independent holonomy vectors, then $\Kcr(S) = \QQ$ and the 
inclusion $\Kcr(S) \subseteq \Ksc(S)$ trivially holds.\\
Since the Veech group preserves the set of holonomy vectors, we furthermore 
have that if there are at least two linearly independent holonomy vectors, 
then $\Ktr(S) \subseteq \Khol(S)$. However, if all holonomy vectors are parallel, 
it is not in general 
true that $\Ktr(S) \subseteq \Khol(S)$.
An example of a surface $S$ showing this is given in \cite[Lemma 3.7]{PSV}:
The surface $S$
is obtained from glueing two copies of $\RR^2$ along horizontal
slits $l_n$ of the plane with end points $(4n+1,0)$ and $(4n+3,0)$.
In particular all saddle connections are horizontal and
the fields $\Khol(S)$, $\Kcr(S)$ and $\Ksc(S)$ are all $\Q$.
But the Veech group is very big. It consists of
all matrices in $\mathbf{GL}_{+}(2,\R)$ which fix the first
standard basis vector $e_1$; compare  \cite[Lemma 3.7]{PSV}.

\begin{remark}\label{uncountable-vg-examples}
The translation surface $S$ from \cite[Lemma 3.7]{PSV} has the following properties:
\[\Gamma(S) = \left\{ \bpm 1&t\\0&s\epm| t \in \R, s \in \R_{+}\right\}\]
and $\Khol(S) = \Kcr(S) = \Ksc(S) = \Q$.
In particular, we have   $\Ktr(S) = \RR$.
\end{remark}
 
Finally, in \Cref{crsc} we see
that for a large class of translation surfaces we have that
$\Kcr(S) = \Ksc(S)$. The main argument of the proof was given in \cite{GJ}
for precompact surfaces.

\begin{proposition}\label{crsc}
Let $S$ be a (possibly non precompact) tame 
translation surface, $\compl{S}$ its metric completion and 
$\Sigma \subset \compl{S}$ its set of singularities. 
Suppose that $\compl{S}$ has a geodesic
triangulation by countably many triangles $\Delta_k$ 
($k \in I$ for some index set $I$) such that
the set of vertices equals $\Sigma$.
We then have {\em $\Kcr(S) = \Ksc(S)$}.
\end{proposition}

\begin{proof}
The inclusion ''$\subseteq$'' was shown in 
general in the paragraph below \Cref{remark3.3}.
The inclusion ''$\supseteq$''
follows from  \cite[Proposition 5.2]{GJ}. The statement 
there is for precompact surfaces, but the proof works in the
same way if there exists a 
triangulation as required in this proposition.
More precisely, in \cite{GJ} it is shown that 
the $\Kcr(S)$-vector space $V(S)$ spanned 
by the image of $H_1(\compl{S},\Sigma; \ZZ)$ under 
the holonomy map is \mbox{2-dimensional}
over $\Kcr(S)$. Hence $\Ksc(S) \subseteq \Kcr(S)$.
\end{proof} \qed 

It follows from \Cref{rel} that in general
no further inclusions between
the four fields from \Cref{definition-fields} hold than those stated above; 
see \Cref{corrolary-relations-between-fields} for a subsumption of the 
relations between the fields.\\

\indent If $S$ is a precompact translation surface of genus $g$, then $[\Ktr(S):\Q]\leq g$. 
Moreover, the traces of elements in $\Gamma(S)$ are algebraic integers (see \cite{Mc}). 
When dealing with tame translation surfaces, such algebraic properties do not hold in general.
\begin{proposition}
  \label{alg}
For each $n\in\N\cup\{\infty\}$ there exists a tame translation surface $S_n$ of infinite 
genus such that the transcendence degree of the field extension $\emKtr(S_n)/\Q$ 
is $n$. $S_n$ can be chosen to be a Loch Ness monster. 
\end{proposition}
\begin{proof}
Let $\{\lambda_1, \ldots, \lambda_n\}$ be $\Q$-algebraically independent real numbers 
with $|\lambda_i| > 2$.
Define
$$
G_n:=\left<\left(
\begin{array}{cc}
\mu & 0 \\
0 & \mu^{-1}
\end{array}
\right)
\hspace{1mm} | \hspace{2mm} \mu+\mu^{-1} = \lambda_i \mbox{ with } i \in \{1,\ldots, n\}\right >.
$$
$G_n$ is countable and a subgroup of the diagonal group. 
In particular, $G_n$ is disjoint from the set $\mathcal{U}$
of contraction matrices; compare 
\Cref{P:PSV} for the definition of $\mathcal{U}$. 
Thus, we can apply \Cref{P:PSV}
and obtain a surface $S_n$ with Veech group $G_n$.
We have
\[\Q \subset \Q(\lambda_1, \ldots, \lambda_n) \subseteq \Ktr(S_n) \subseteq L = \Q(\mu | 
\mu + \mu^{-1} = \lambda_i \mbox{ with } i \in \{1, \ldots, n\}).\]
Since the generators $\mu$ of $L$ are algebraic over $\Q(\lambda_1, \ldots, \lambda_n)$,
it follows that $L/\Q(\lambda_1, \ldots, \lambda_n)$ and thus also $\Ktr(S_n)/\Q(\lambda_1, \ldots, \lambda_n)$ is algebraic
and we obtain the claim.
\qed 
\end{proof}

\indent If $\mu_1$ is one of the two solutions of $ \mu+\mu^{-1}=\pi$ and  
$$
  G:=\left<\left(
  \begin{array}{cc}
    \mu_1 & 0 \\
    0 & \mu_1^{-1}
  \end{array}
  \right)
  \hspace{1mm}\right>,
$$
then we obtain in the same way the following corollary.
\begin{corollary}
 There are examples of tame translation surfaces $S$ of infinite genus 
with a cyclic hyperbolic Veech group such that $\emKtr(S)$ is not a number field. 
Again the translation surface can be chosen
as a Loch Ness monster.
\end{corollary}
\indent Transcendental numbers naturally appear also in fields associated with Veech groups arising from 
a  generic triangular billiard. Indeed, let $\mathcal{T}\subset\R^2$ denote the space 
of triangles parametrised by two angles $(\theta_1,\theta_2)$. 
Remark that $\mathcal{T}$ is a simplex.  
For every $T = T_{(\theta_1,\theta_2)}\in\mathcal{T}$, a classical construction 
due to Katok and Zemljakov  produces 
a tame flat surface $S_T$ from $T$ \cite{KZ}. If $T$ has an interior angle which is not 
commensurable with $\pi$, $S_T$ is a Loch Ness monster; compare \cite{V1}. 
\begin{proposition}
The set $\mathcal{T}'\subset \mathcal{T}$ formed by those triangles such that $\emKsc(S_T)$, $\emKcr(S_T)$ 
and $\emKtr(S_T)$ are not number fields, is of total (Lebesgue) measure in $\mathcal{T}$.
\end{proposition}
\begin{proof}
Since $S_T$ has a triangulation with countably many triangles satisfying the hypotheses of 
\Cref{crsc}, the fields $\Ksc(S_T)$ and $\Kcr(S_T)$ coincide. 
Without loss of generality we can assume that the triangle $T = T_{(\theta_1,\theta_2)}$
has the vertices $0$, $1$ and $\rho e^{i\theta_1}$ (with $\rho > 0$) in the complex plane $\C$.
When doing the Katok-Zemljakov construction we start by reflecting $T$ at its
edges. Thus in particular $S_T$ contains the geodesic quadrangle shown in \Cref{figure-quadrangle}.
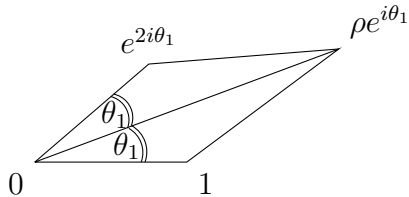
\begin{figure}[h]
  \begin{center}
    \setlength{\unitlength}{3cm}
    \begin{tikzpicture}
      \draw (0,0) node[anchor = north east]{$0$} -- (2,0) node[anchor = north west]{$1$} 
               -- (4,1.5) node[anchor = south west]{$\rho e^{i\theta_1}$} -- (0,0);
      \draw (0,0) -- (1.5, 1.3) node[anchor = south]{$e^{2i\theta_1}$} -- (4,1.5);
      \draw (1.4,0) arc (-20:60:.37)  node at (10:1.23) {$\theta_1$} node at (31:1.23) {$\theta_1$};      
      \draw (1.45,0) arc (-20:60:.4); 
      \draw (1.22,.45) arc (-20:75:.32);
      \draw (1.27,.48) arc (-20:78:.32);
      \end{tikzpicture}
    \caption{Geodesic quadrangle on the surface $S_{T}$ with $T$ the triangle $T_{(\theta_1,\theta_2)}$}
    \label{figure-quadrangle}
  \end{center}
\end{figure}


Thus the vectors $v_1 = (1,0)$, $v_2 = (\rho\cos\theta_1, \rho\sin\theta_1)$ 
and $v_3 = (\cos 2\theta_1, \sin 2\theta_1)$ 
are holonomy vectors.
Choose $\{v_1,v_2\}$ as basis of $\R^2$. We then have 
$v_3 = a\cdot v_1 + b\cdot v_2$
with
$$
a=-1 \mbox{ and } b=\frac{2\cos\theta_1}{\rho}.
$$
\indent Therefore $\frac{2\cos\theta_1}{\rho}$ is an element of $\Ksc(S) = \Kcr(S)$. 
Furthermore, from \cite{V}
we know that the matrix representing the rotation by $\theta_1$
is in $\Gamma(S_T)$. Hence $2\cos\theta_1$ is in $\Ktr(S_T)$. 
Thus if we choose the values $\frac{\cos\theta_1}{\rho}$, respectively $\cos\theta_1$,
to be non algebraic numbers, then $\Ksc(S) = \Kcr(S)$, respectively $\Ktr(S_T)$, are not number fields. \qed
\end{proof}


\section{Proof of main results}
\label{proofs}
In this section we prove the results stated in the introduction.\\

\emph{Proof \Cref{mainthm}}. We begin proving part (\labelcref{Thm1i}).  
Let $T_0 = T\setminus \{\infty\}$ be the once punctured torus with
$T = \RR^2/L$, where $L$ is a lattice in $\RR^2$,
and with $\infty \in T$ the image of the origin removed.
Let  $p:S\longrightarrow T_0$ be an unramified translation covering.
The existence of such a covering is equivalent to $S$ being affine equivalent to an origami. We use the notation
from  \Cref{prelim}. In particular $\pi_S:\widetilde{S}\longrightarrow S$
is a universal cover, $\overline{\widetilde{S}}$ and $\overline{S}$
are the metric completions of $\widetilde{S}$ and $S$, respectively, and
$\widetilde{\Sigma}(S)$ is the set of developed cone points.
We then have that the following diagram commutes, since $p$ and $\pi_S$ are translation maps:
\begin{equation}
  \xymatrix{
    \widetilde{S}\hspace{1mm} \ar[d]^{\pi_S}\ar@{^(->}[r]&\overline{\widetilde{S}}\ar[d]^{\overline{\pi_S}}\ar[dr]^{\overline{dev}}&\\
    S\hspace{1mm}\ar@{^(->}[r]\ar[d]^{p}&\ar[d]^{\overline{p}}\overline{S} & \R^2\ar[dl]^{\pi_T}\\
    T_0\hspace{1mm}\ar@{^(->}[r]& T & \hspace*{10mm}.
  }
\end{equation}
Given that $T\setminus T_0=\infty=\overline{p}\circ\overline{\pi_S}(\overline{\widetilde{S}}\setminus{\widetilde{S}})$, 
the projection of $\widetilde{\Sigma}(S)$ to $T$ is just a point. This proves sufficiency.\\
\indent \Cref{ehol} implies that if $\widetilde{\Sigma}(S)$ is contained in $L+x$ 
then every $\rm hol(\gamma)$ is a translation of the plane of the form $z\to z+\lambda_{\gamma}$, where $\lambda_{\gamma}\in L$. 
Puncture $\widetilde{S}$ and $S$ at ${\rm dev}^{-1}(L+x)$ and $\pi_S({\rm dev}^{-1}(L+x))$ respectively to obtain 
$\widetilde{S}_0$ and $S_0$ and denote $\R^2_0=\R^2\setminus (L+x)$. Let $\pi_{S\mid}: \widetilde{S}_0 \longrightarrow S_0$ and 
$\pi_{T\mid}: \R^2_0 \longrightarrow T_0$ be the restrictions of the universal covers
$\pi_S$ and $\pi_T$.
 Given that $\widetilde{S}_0$ 
has the translation structure induced by pull-back of $\pi_{S\mid}$, 
the map ${\rm dev}_|:\widetilde{S}_0\longrightarrow \R^2_0$ is a flat surjective 
map; compare \cite[\S 3.4]{Thu}. \Cref{hol} implies that
\begin{equation}
  \label{cover}
  \xymatrix{
    \widetilde{S}_0 \ar[r]^{\rm dev_|} \ar[d]^{\pi_{S\mid}} & \R^2_0 \ar[d]^{\pi_{T\mid}}\\
    S_0 & T_0}
\end{equation}
descends to a flat covering map $p:S_0\longrightarrow T_0$. 
Hence $\overline{S}=\overline{S_0}$ defines a covering over a 
flat torus ramified at most over one point. This proves necessity.\\
\\
Now we prove part (\labelcref{partB}). First we prove that every origami satisfies conditions 
(\labelcref{vg}), (\labelcref{Kcr}), (\labelcref{Khol}), (\labelcref{Ksc}) and (\labelcref{Ktr}). \\
Let $\overline{p}:\overline{S} \to T$ be an origami ramified at most over $\infty \in T$.
All saddle connections of $S$ are preimages of closed simple curves 
on $T$ with a base point at $\infty$. This implies that all holonomy vectors have integer coordinates.
Thus $\Ksc(S) = \Khol(S) = \Kcr(S)  = \QQ$. 
Hence every origami fulfils conditions (\labelcref{Kcr}), (\labelcref{Khol}) and 
(\labelcref{Ksc}) in part (\labelcref{partB}). 
If $S$ furthermore has at least two linearly independent holonomy vectors, 
then the Veech group must preserve the lattice spanned by them.
Thus it is commensurable to a (possible infinite index)
subgroup of SL$(2,\ZZ)$ and $S$ fulfils in addition (\labelcref{vg}) and
(\labelcref{Ktr}).\\ 

We finally prove that none of the conditions in theorem (\labelcref{vg}) to (\labelcref{Ktr}) 
imply that $S$ is an origami. \Cref{E:1} shows that neither (\labelcref{vg}) nor (\labelcref{Ktr}) 
imply that $S$ is an origami. \Cref{E:2} shows that neither (\labelcref{Kcr}), nor (\labelcref{Khol}),
nor (\labelcref{Ksc}) imply that  $S$ is an origami.

\begin{example}
  \label{E:1}
In this example we construct a tame translation surface $S$ whose Veech group $\Gamma(S)$ is ${\rm SL}(2,\Z)$, hence  $\Ktr(S)=\Q$, but which is not an origami. We achieve this making a slight modification of the construction presented in \Cref{subs:PSV}. Let $G={\rm SL}(2,\Z)$. Apply the construction  in \Cref{subs:PSV}
to $G$ but choose the family of marks
$\mathcal{C}^{-1}$ in such a way that the there exists $N\in\Z$
and irrational $\alpha>0$ so that $(\alpha,N)$ is a holonomy vector. This is possible since in 
the cited construction the choice of the point $(x_1,y_1)$ is free. Observe that $v_1=(-1,1)$, $v_2=(0,1)$, 
$v_3=(1,0)$ and $v_4=(\alpha,N)$ are holonomy vectors of $\overline{S}$. Let 
$l_i$ be lines in $\PP^1(\R)$ containing $v_i$, $i=1,\ldots,4$ respectively. 
A direct calculation shows that the cross ratio of these four lines is $\frac{\alpha}{\alpha+N}$, which lies in $\Kcr(S)$. Hence $\Kcr(S)$ is not isomorphic to $\Q$ and therefore $S$ cannot be an origami.
\end{example}
\begin{example}
  \label{E:2}
In this example we construct a surface $S$ whose Veech group is not a discrete subgroup of ${\rm SL}(2,\R)$ 
(hence $S$ cannot be an origami, since in addition $S$ has two non parallel saddle connections) 
but such that 
\begin{equation}
  \label{E:allQ}
\Kcr(S) = \Khol(S) = \Ksc(S) = \Ktr(S) = \QQ.
\end{equation}
Consider $G={\rm SL}(2,\Q)$ or $G={\rm SO}(2,\Q)$. 
These are non-discrete countable subgroups of ${\rm SL}(2,\R)$ with no contracting elements. Hence we can apply the construction from \Cref{P:PSV} to $G$ but choosing the points $(x_i,y_i)$ that define the families of marks 
$\mathcal{C}^i$ in $\Q\times\Q$ for all $i\geq 1$ indexing a countable set of generators of $G$. 
The result is a tame translation surface $S$ whose Veech group is isomorphic to $G$ and whose holonomy vectors $S$ have all coordinates in $\Q\times\Q$. This implies \labelcref{E:allQ}. 
\end{example} \qed
\emph{Proof \Cref{rel}}. 
Let us first show that (\labelcref{Thm2i}) holds. The tame translation surface $S$ in \Cref{E:1} is such that $\Gamma(S)={\rm SL(2,\Z)}$ and $\Kcr(S)$ is not isomorphic to $\Q$. Since in general $\Kcr(S)$ is a subfield of $\Ksc(S)$ this surface also satisfies that $\Gamma(S)={\rm SL(2,\Z)}$ and $\Ksc(S)$ is not isomorphic to $\Q$.
To finish the proof of (\labelcref{Thm2i}) we consider the following example.
\begin{example}
  \label{E:3}
In this example we construct a tame translation surface such that $\Gamma(S)={\rm SL(2,\Z)}$ and $\Khol(S)$ is not isomorphic to $\Q$. Apply the construction described in \Cref{subs:PSV} to $G={\rm SL}(2,\Z)$ but consider the following modification. Let $\{e_1,e_2\}$ be the standard basis of $\R^2$. There exists a natural number $n>0$ such that the mark $M$ in $A_{Id}$ 
whose end points are $-ne_1$ and $-(n-1)e_1$ does not intersect all other marks used in the construction. 
On a $[0,\pi]\times [0,e]$ rectangle $R$, where $e$ is Euler's number, identify 
opposite sides to obtain a flat torus $T$. Consider on $T$ a horizontal mark $M'$ of length 1 and glue $A_{Id}$ with $T$ along $M$ and $M'$. Then proceed with the construction in a ${\rm SL}(2,\Z)$-equivariant way. This produces a tame translation surface $S$ whose Veech group is ${\rm SL}(2,\Z)$. The image of $H_1(\overline{S},\Z)$ under the holonomy map contains the vectors $e_1$, $e\cdot e_1$ and $\pi\cdot e_2$. Hence $\Khol(S)$ is not isomorphic to $\Q$.
\end{example}
Part (\labelcref{Thm2ii}) follows from \Cref{E:2}.
\indent We now prove (\labelcref{Thm2Part3}). First we construct $S$ such that $\Kcr(S)=\Q$ but $\Ksc(S)$ is not. Consider the following example.
\begin{example}
  \label{E:4}
Let  $P_1$, $P_2$ and $P_3$ be three copies of $\R^2$; choose on each copy an origin, 
and let $\{e_1,e_2\}$ be the standard basis. Consider the following: 
\begin{enumerate}
 \item Marks $v_n$ on the plane $P_1$ along segments whose end points are $n\cdot e_2$ and $n\cdot e_2+e_1$ with $n=0,1$.
\item Marks on $P_2$ and $P_3$ along the segments $w_0$, $w_1$ whose end points are $(0,0)$ and $(1,0)$,
and then along the segments $z_0$ and $z_1$ whose end points are $(2,0)$ and $(2+\sqrt{p},0)$, for some prime $p$.
\end{enumerate}
Glue the three planes along slits as follows: $v_i$ to $w_{i}$, for $i=0,1$ and $z_0$ to $z_1$. The result is a surface $S$ for which $\{0,1,-1,\infty\}$ parametrises all possible slopes of lines through the origin in $\R^2$ containing holonomy vectors of saddle connections. Hence $\Kcr(S)=\Q$. On the other hand, the set of holonomy vectors contains $(1,0)$, $(0,1)$, $(1,1)$ and $(\sqrt{p},0)$. Therefore $\Ksc(S)$ contains $\Q(\sqrt{p})$ as a subfield. \\    
\end{example}
\indent We finish the proof of (\labelcref{Thm2Part3}) by constructing a precompact tame translation surface such that $\Khol(S)=\Q$, but $\Ksc(S)$ is not. Consider the following example.
\begin{example}
  \label{E:5}
Consider two copies $L_1$ and $L_2$ of the L-shaped origami 
tiled by three unit squares; see \emph{e.g.} \cite[Example on p. 293]{HL2}. Consider a point $p_i\in L_i$ 
at distance $0<\varepsilon<<1$ from the $6\pi$-angle singularity $s_i$, $i=1,2$. 
Let $m_i$ be a marking of length $\varepsilon$ on $L_i$ defined by a geodesic of length $\varepsilon$ 
joining $p_i$ to $s_i$, $i=1,2$. We can choose $p_i$ so that 
both markings are parallel
and the vector defined by them has irrational coordinates. Glue then $L_1$ and $L_2$ along $m_1$ and $m_2$ to obtain $S$. 
By construction  $h(H_1(\overline{S},\Z))=\Z\times\Z$, hence $\Khol(S)=\Q$, 
but $h(H_1(\overline{S},\Sigma;\Z))$ contains an orthonormal base $\{e_1,e_2\}$ 
and a vector $h(m_1)$ with irrational coordinates.
This implies that $\Ksc(S)$ is not isomorphic to $\Q$.\\  
\end{example}

\indent We address (\labelcref{Thm2iv}) now. Observe that the surface $S$ constructed 
in \Cref{E:5} satisfies that $\Khol(S)=\Q$ but $\Kcr(S)$ is not equal to $\Q$. 
Indeed, we have saddle connections of slope $0$, $1$ and $\infty$.
Since the slope of $h(m_1)$ is irrational, we are done. 
We now construct $S$ such that $\Kcr(S)=\Q$, but $\Khol(S)$ is not. 
We will furthermore have that $S$ has four pairwise nonparallel
holonomy vectors thus $\Kcr(S)$ is not trivially $\Q$.


\begin{example}
  \label{E:6}
  Take two copies of the real plane $P_1$ and $P_2$. Choose an origin and let ${e_1,e_2}$ be the standard basis. 
Let $\mu_i>1$, $i=1,2,3$ be three distinct irrational numbers and define $\lambda_0=0$ and $\lambda_n=\sum_{i=1}^n\mu_i$ for $n=1,2,3$. 
On $P_1$ consider the markings $m_n$ whose end points are $n e_2$ and $n e_2+ e_1$ for $n=0,\ldots,3$. On $P_2$ consider the markings $m_n'$ whose end points are $(n+\lambda_n)e_1$ and $(n+\lambda_n+1)e_1$ for $n=0,\ldots,3$. 
Glue $P_1$ and $P_2$ along the markings $m_n$ and $m_n'$. The result is a tame flat surface $S$ with eight $4\pi$-angle singularities. These singularities lie on $P_2$ on a horizontal line, and hence we can naturally order them from, say, left to right. Let us denote these ordered singularities by $a_j$, $j=1,\ldots,8$. Let $g_{e_1}(a_i,a_j)$ (respectively $g_{e_2}(a_i,a_j)$) be the directed geodesic in $S$ parallel to $e_1$ (respectively $e_2$) joining $a_i$ with $a_j$. Define in $H_1(\overline{S},\Z)$
\begin{itemize}
\item the cycle $c_1$ as $g_{e_1}(a_3,a_4)g_{e_1}(a_4,a_5)g_{e_2}(a_5,a_3)$,
\item the cycle $c_2$ as $g_{e_1}(a_4,a_3)g_{e_1}(a_3,a_2)g_{e_2}(a_2,a_4)$,
\item the cycle $c_3$ as $g_{e_2}(a_6,a_8)g_{e_1}(a_8,a_7)g_{e_1}(a_7,a_6)$.
\end{itemize}
Where the product is defined to be the composition of geodesics on $S$. 
Note that $h(c_1)=(1+\mu_2,-1)$, $h(c_2)=(-(1+\mu_1),1)$ and $h(c_3)=(-(1+\mu_3),1)$. 
We can choose parameters $\mu_i$, $i=1,2,3$ so that the $\Z$-module generated by these 3 vectors has rank 3. Therefore $\Khol(S)$ cannot be isomorphic to $\Q$.\\
\end{example}

\indent We address now (\labelcref{Thm2v}). We construct first a flat surface $S$ for which $\Ktr(S)=\Q$ but none of the conditions (\labelcref{vg}), (\labelcref{Kcr}), (\labelcref{Khol}) 
or (\labelcref{Ksc}) in \Cref{mainthm} hold. We achieve this by making a slight modification on the construction of the surface in \Cref{E:3} in the following way. First, change ${\rm SL}(2,\Z)$ for ${\rm SL}(2,\Q)$. Second, let the added mark M be of unit length and such that the vector defined by developing it along the flat structure neither  lies 
in the lattice $\pi\Z\times e\Z$ nor has rational slope. The result of this modification is a tame translation surface $S$ homeomorphic to the Loch Ness monster for which $\Gamma(S)={\rm SL}(2,\Q)$ and for which both $\Kcr(S)$ and $\Khol(S)$ (hence $\Ksc(S)$ as well) have transcendence degree at least 1 over $\Q$.\\
Finally, an example of a surface $S$  which satisfies (\labelcref{Khol}), (\labelcref{Ksc}) 
and (\labelcref{Kcr}),  but with $\Ktr(S)\neq\Q$, is given  in \cite[Lemma 3.7]{PSV}; 
see \Cref{uncountable-vg-examples}. We underline that all holonomy vectors in this surface $S$ are parallel and hence $\Kcr(S)$ is by definition isomorphic to $\Q$. For the sake of completeness we construct a tame translation surface $S$ where not all holonomy vectors are parallel, such $\Kcr(S)=\Q$, but where $\Ktr(S)$ is not equal to $\Q$. Let $E_0$ be a copy of the affine plane $\R^2$ with a chosen origin and $(x,y)$-coordinates. Slit $E_0$ along the rays $R_v:=(0,y\geq 1)$ and $R_h:=(x\geq 1,0)$ to obtain $\hat{E}_0$. 
Choose an irrational $0<\lambda<1$  and $n\in\N$ so that $1<n\lambda$. Define
 $$
 M:=\left( \begin{array}{cc}
\lambda & 0  \\
0 & n\lambda  
\end{array} \right) \hspace{1cm} R^k_v:=M^kR_v\hspace{1cm} R^k_h:=M^kR_h\hspace{5mm}k\in\Z .
$$
Here $M^k$ acts linearly on $E_0$. For $k\neq 0$, slit a copy of $E_0$ along the rays $R_v^k$ and $R_h^k$ to obtain $\hat{E}_k$. 
We glue the family of slitted planes $\{\hat{E}_k\}_{k\in\Z}$ to obtain the desired tame flat surface 
as follows. Each $\hat{E}_k$ has a ``vertical boundary" formed by two vertical rays issuing 
from the point of coordinates $(0,(n\lambda)^k)$. Denote by $R_{v,l}^k$ and $R_{v,r}^k$ the 
boundary ray to the left and right respectively. Identify by a translation the 
rays $R_{v,r}^k$ with $R_{v,l}^{k+1}$, for each $k\in\Z$. Denote by $R_{h,b}^k$ and $R_{h,t}^k$ the horizontal boundary rays in $\hat{E}_k$ to the bottom and top respectively. Identify by a translation $R_{h,b}^k$ with $R_{h,t}^{k+1}$ for each $k\in\Z$.\\
\indent By construction, $\{(-\lambda^k,(n\lambda)^k)\}_{k\in\Z}$ is the set of all holonomy vectors of $S$. Clearly, 
all slopes involved are rational; hence $\Kcr(S)=\Q$. On the other hand, $M\in\Gamma(S)$ and $tr(M)=(n+1)\lambda$. 
Note that the surface $S$ constructed in this last paragraph admits no triangulation satisfying the hypotheses of \Cref{crsc}.\qed

\begin{corollary}\label{corrolary-relations-between-fields}
  Between the four fields $\emKtr(S)$, $\emKhol(S)$, $\emKcr(S)$ and $\emKsc(S)$
  the following relations hold:
  \begin{enumerate}
  \item\label{cor:fields:i}
    $\emKhol(S) \subseteq \emKsc(S)$ and $\emKcr(S) \subseteq \emKsc(S)$.
  \item\label{cor:fields:ii}
    For each other pair $(i,j)$, with $i,j \in \{\rm{tr, hol, cr, sc}\}$,
    $(i,j) \neq (\rm{hol},\rm{sc})$ and $(i,j) \neq (\rm{cr},\rm{sc})$,
    we can find surfaces $S$ such that $K_i(S) \not\subseteq K_j(S)$.
    In these examples we can always choose $K_j(S)$ to be $\Q$.
  \item\label{cor:fields:iii}
    If $S$ has two non parallel holonomy vectors, then $\emKtr(S) \subseteq \emKhol(S)$. 
  \item\label{cor:fields:iv}
    If $\overline{S}$ has a geodesic triangulation by countably many triangles  
    whose vertices form the set $\Sigma$ of singularities of $S$, then
    $\emKcr(S) = \emKsc(S)$.
  \end{enumerate}
\end{corollary}

\begin{proof}
(\labelcref{cor:fields:i}) is shown in \Cref{section3} before \Cref{uncountable-vg-examples};
(\labelcref{cor:fields:ii}) is shown in \Cref{rel} and in \Cref{uncountable-vg-examples};
(\labelcref{cor:fields:iii}) is shown before  \Cref{uncountable-vg-examples} and
(\labelcref{cor:fields:iv}) is the result of \Cref{crsc}.\qed
\end{proof}
{\em Proof \Cref{OUTF2}:}

Let $\Gamma$ be a subgroup of $\sltwo(\ZZ)$. By \Cref{P:PSV} we know
that there is a translation surface $S$ with Veech group $\Gamma$. Furthermore
in the construction all slits can be chosen such that their end points
are integer points in the corresponding plane $\CC = \RR^2$; thus
$S$ is an origami by \Cref{mainthm}, part (\labelcref{Thm1i}). Hence it allows 
for a subset $S^*$ of $S$, whose complement is a discrete set of points, 
an unramified covering $p:S^* \to T_0$ to the once puncture unit torus $T_0$.
Recall that $p$ defines the conjugacy class $[U]$ of a subgroup $U$ of $F_2$ as follows.
Let $U$ be the fundamental group of $S^*$. 
It is
embedded into $F_2 = \pi_1(T_0)$
via the homomorphism $p_*$ between fundamental groups induced by $p$.
The embedding depends on the choices of the base points up to conjugation. 
In \cite{S} this is used to give the description
of the Veech group  
completely in terms of $[U]$; compare Theorem~\labelcref{thmc}. Recall for this that 
the outer automorphism group $\out(F_2)$ is isomorphic to $\gltwo(\ZZ)$.
Furthermore it naturally acts on the set of the conjugacy classes
of subgroups $U$ of $F_2$. 

\begin{theoremothers}[\cite{S}, Prop. 2.1]\label{thmc}
  The Veech group $\Gamma(S^*)$ equals the stabiliser of the conjugacy
  class $[U]$ in $\sltwo(\ZZ)$ under the action described above.
\end{theoremothers}
The theorem in \cite{S} considers only finite origamis, but the proof
works in the same way for infinite origamis. Recall furthermore that
$\Gamma(S^*) = \Gamma(S) \cap \sltwo(\ZZ)$ and 
$\Gamma(S) \subseteq \sltwo(\ZZ)$ if and only if the $\ZZ$-module 
spanned by the holonomy vectors of the saddle connections equals $\ZZ^2$.
We can easily choose the marks in the construction in \cite{PSV}
such that this condition is fulfilled.\\ \qed

{\em Proof \Cref{ncon}:}
First notice that the translation surfaces constructed in the proof of \Cref{rel} parts (\labelcref{Thm2i}) 
and (\labelcref{Thm2v}) are both counterexamples for 
statements (\labelcref{Thm3Part1}) and (\labelcref{Thm3Part2}). 
Furthermore, \Cref{alg} shows that two hyperbolic elements in $ \Gamma(S)$ do not have 
to generate the same trace field.
To disprove (\labelcref{Thm3Part3}) 
we let $\mu$ be a solution to the equation $\mu+\mu^{-1}=\sqrt[3]{11}$
and $G$ is the group generated by the matrices
\begin{equation}
  \label{ntr}
  \left( \begin{array}{cc}
    1 & 1  \\
    0 & 1  
  \end{array} \right), \hspace{5mm}
  \left( \begin{array}{cc}
    1 & 0  \\
    1 & 1  
  \end{array} \right)\hspace{5mm}
  \text{and}\hspace{5mm}
  \left( \begin{array}{cc}
    \mu & 0  \\
    0 & \mu^{-1}  
  \end{array} \right) ,
\end{equation}
then  \Cref{P:PSV} produces a tame translation surface $S$
with Veech group $G$ for which $\Ktr(S) = \QQ(\sqrt[3]{11})$ 
is not totally real and thus is a counterexample for (\labelcref{Thm3Part3}).\\
\indent Finally for disproving (\labelcref{Thm3Part4}) we construct a 
tame translation surface $S$ with a hyperbolic element in its Veech group for which $\Lambda$ 
has infinite index in $\Lambda_0$. The construction has two steps.\\
\emph{Step 1}: Let $M$ be the matrix given by
\begin{equation}
  \label{m2}
  \left( \begin{array}{cc}
    2 & 0  \\
    0 & \frac{1}{2}  
  \end{array} \right).
\end{equation}

Let $S'$ be the tame translation surface obtained from \Cref{P:PSV}
for the group $G'$ generated by $M$. Let $\Lambda'$ be the image in $\R^2$ under the holonomy map of $H_1(\overline{S'},\Z)$, $\{e_1,e_2\}$ be the standard basis of $\R^2$ and $\beta:= G' \cdot\{e_1,e_2\}$. 
We suppose without loss of generality that $e_1$ and $e_{2}$
lie in $\Lambda'$. \\
\emph{Step 2}: Let $\alpha=\{v_j\}_{j\in\N}\subset\R^2\setminus \Lambda'$ be a sequence of $\Q$-linearly independent vectors. 
We modify the construction in \Cref{P:PSV} (applied to $G'$) in the following way. 
We add to the page $A_{Id}$ a family of marks parallel to vectors in $\alpha$. 
We can suppose that the new marks lie in the left-half plane $Re(z)<0$ in $A_{Id}$ and
are disjoint by pairs and do not intersect any of the marks in $\mathcal{C}_1$ used in the
construction from Step 1. For each $j\in\N$ there exists a natural number $k_j$ such that $2k_j>|v_j|$.  
Let $T_j$ be the torus obtained from a $2k_j\times 2k_j$ square by identifying opposite sides. 
Slit each $T_j$ along a vector parallel to $v_j$ and glue it to $A_{Id}$ along the mark parallel to $v_j$. 
Denote by $A_{Id}'$ the result of performing this operation for every $j\in\N$, 
then proceed just the same construction as in \Cref{P:PSV}. Let $S$ be the resulting translation surface.
Observe that glueing in the tori $T_j$ produces new elements in $H_{1}(\overline{S},\Z)$ whose image under the holonomy map lie in $\Z\times\Z$. Thus the subgroups of $\R^{2}$ generated by the image under the holonomy map of $H_1(\overline{S}',\Z)$
and $H_1(\overline{S},\Z)$ are the same.
\indent Let $\Lambda$ be the image in $\R^2$ under the holonomy map of $H_1(S,\Z)$. By construction, the index of $\Lambda$ in $\Lambda_{0}$ is at least the cardinality of $\alpha$, which is infinite.\qed

\end{document}